\documentclass[reqno,a4paper,12pt]{amsart}
\usepackage{latexsym, amsmath, amssymb, amsthm, a4}

\usepackage{amsfonts, bbold, bbm, marvosym}

\usepackage{mathrsfs}

\usepackage[mathscr]{eucal}
\font\sc=rsfs10 at 12pt

\renewcommand{\a}{\alpha}

\newcommand{\g}{\gamma}
\newcommand{\G}{\Gamma}

\newcommand{\e}{\epsilon}

\renewcommand{\i}{\iota}

\renewcommand{\l}{\lambda}

\newcommand{\m}{\mu}

\newcommand{\x}{\xi}
\newcommand{\X}{\Xi}
\renewcommand{\r}{\rho}
\newcommand{\s}{\sigma}

\newcommand{\vf}{\varphi}
\newcommand{\F}{\Phi}

\newcommand{\R}{{\mathbb R}}

\newcommand{\T}{{\mathbb{T}}}

\newcommand{\ab}{{\mathbf a}}

\newcommand{\kb}{{\mathbf k}}

\newcommand{\mb}{{\mathbf m}}
\newcommand{\nb}{{\mathbf n}}

\newcommand{\rb}{{\mathbf r}}

\newcommand{\tb}{{\mathbf t}}

\newcommand{\Ab}{{\mathbf A}}

\newcommand{\Cb}{{\mathbf C}}

\newcommand{\Hb}{{\mathbf H}}

\newcommand{\Kb}{{\mathbf K}}

\newcommand{\Mb}{{\mathbf M}}
\newcommand{\Nb}{{\mathbf N}}

\newcommand{\Qb}{{\mathbf Q}}

\newcommand{\Sbb}{{\mathbf S}}
\newcommand{\Tb}{{\mathbf T}}
\newcommand{\Ub}{{\mathbf U}}

\newcommand{\AF}{\mathfrak A}

\newcommand{\dF}{\mathfrak d}

\newcommand{\kF}{\mathfrak k}
\newcommand{\KF}{\mathfrak K}

\newcommand{\LF}{\mathfrak L}

\newcommand{\Ac}{{\mathcal A}}
\newcommand{\Bc}{{\mathcal B}}

\newcommand{\Fc}{{\mathcal F}}

\newcommand{\Hc}{{\mathcal H}}

\newcommand{\Mcc}{{\mathcal M}}

\newcommand{\Pc}{{\mathcal P}}

\newcommand{\Uc}{{\mathcal U}} 

\newcommand{\Ks}{\sc\mbox{K}\hspace{1.0pt}}

\newcommand{\diam}{\operatorname{diam\,}}

\newcommand{\Tr}{\operatorname{Tr\,}}

\newcommand{\beq}{\begin{equation}}
\newcommand{\eeq}{\end{equation}}

\numberwithin{equation}{section}
\numberwithin{figure}{section}

\newtheorem{thm}{Theorem}[section]
\newtheorem{cor}[thm]{Corollary}

\newtheorem{proposition}[thm]{Proposition}
\newtheorem{remark}[thm]{Remark}

\theoremstyle{definition}

\begin{document}

\title[Integral operators]{Spectral estimates and asymptotics  for integral operators on singular sets}

\author{Grigori Rozenblum }

\address{Chalmers Univ. of Technology, Sweden, and  The Euler Intern. Math. Institute}
\email{grigori@chalmers.se}
\thanks{The author was supported  by the Ministry of Science and Higher Education of the Russian Federation, Agreement   075--15--2022--287 .}

\author{Grigory Tashchiyan}
\address{St.Petersburg University for Telecommunications, Russia}
\email{gt$\_$47@mail.ru}

\subjclass[2010]{47A75 (primary), 58J50 (secondary)}
\keywords{Singular measures, Integral operators, Spectral  estimates, Eigenvalue asymptotics}
\dedicatory{To Volodya Maz'ya, with admiration}
\begin{abstract}
For  singular numbers  of integral operators of the form $u(x)\mapsto \int F_1(X)K(X,Y,X-Y)F_2(Y)u(Y)\m(dY),$ with measure $\m$ singular with respect to the Lebesgue measure in $\R^\Nb$, order sharp estimates for the counting function are established.  The kernel $K(X,Y,Z)$ is supposed to be smooth in $X,Y$ and in $Z\ne 0$ and to admit an asymptotic expansion in homogeneous functions in $Z$ variable as $Z\to 0.$ The order in estimates is determined by the leading homogeneity order in the kernel and geometric properties of the measure $\m$ and involves integral norms of the weight functions $F_1,F_2$.  For the case of the measure $\m$ being the surface measure of a Lipschitz surface of some positive codimension $\dF,$ in the self-adjoint case, the asymptotics of eigenvalues of this integral operator is found.
\end{abstract}
\maketitle


\section{Introduction}
Since the pathbreaking papers \cite{BS1}, \cite{BS2}, \cite{BS3} by M.Sh. Birman and M.Z. Solomyak in 1960-s and 1970-s,
it became a general wisdom that order sharp eigen- (and singular-) values estimates for various kind of operators under weakest possible regularity conditions, in terms of integral norms of coefficients, open possibilities to the study of finer characteristics of the spectrum, especially, obtaining the  spectral asymptotics. This approach has been demonstrated many times,  including the most recent developments, see \cite{KLP}, \cite{RConnes}, \cite{RT3}. Still, there are certain types of problems   where such estimates are  not known yet. The present note is devoted to one of such missing cases. We consider integral operators with weak diagonal polarity of the kernel acting in the $L_2$ space with measure $\m$ singular with respect to the Lebesgue measure in $\R^\Nb$ and containing integrable weight functions.  Our main results are new spectral estimates for such operators. They are closely related to recent developments in the topic of spectral estimates for pseudodifferential operators with singular measures, see \cite{RSFA}, \cite{RSLap}, \cite{RConnes}, \cite{RT3}, however the specifics of integral operators enables one to reduce the regularity conditions present in  these papers. In the self-adjoint case,  when the measure $\m$ is the surface  measure on a compact Lipschitz surface of a certain codimension $\dF\ge 1$ in $\R^\Nb$, we find the eigenvalue asymptotics.

Our approach to operators involving singular measures, following \cite{RSLap}, \cite{RT3}, is based upon fundamental trace and embedding theorems by V.G. Maz'ya, see \cite{MazBook}. The authors use this occasion to express their admiration with this outstanding mathematician and to wish him many years of active and productive life.

\subsection{Setting}Let $K(X,Y,Z), \, X,Y\in \R^\Nb, Z\in \R^\Nb\setminus 0,$ be an integral kernel of potential type, smooth in   $X,Y$ variables and in $Z$ for $Z\ne 0$. The kernel is supposed to be polyhomogeneous, this means that as $Z\to 0,$ $K(X,Y,Z),$ expands in the asymptotic series
 \begin{equation}\label{Polyhomog}
   K(X,Y,Z)\sim \sum K_{j}(X,Y,Z), \, j=0,\dots, \, Z\to 0.
 \end{equation}
 The function $K_{j}(X,Y,Z)$ is positively homogeneous in $Z$ of degree $\theta +j$; the leading order is
 $\theta>-\Nb$, so the kernel possesses a weak singularity. If $\theta+j$ is an even non-negative number, the symbol $K_j$ may, in addition to the above  homogeneous function which we denote $K^{(hom)}_j$,  contain an additional term $K^{(log)}_{j}(X,Y,Z)$  of the form $Q_j(X,Y,Z)\log|Z|$, with $Q_j(X,Y,Z)$ being a homogeneous polynomial of degree $\theta+j$ in $Z$ variable and smooth in other variables.
 Note that the representation \eqref{Polyhomog} for a given kernel is not unique: it can be changed by making the   Taylor   expansion  in $X-Y$ variable at the point $X$ or $Y$ and re-grouping the resulting terms. By $\Kb$ we denote the integral operator with kernel $K.$

Let $\m$ be a  compactly supported Borel measure on $\R^\Nb$ with support (the smallest closed set of full $\m$-measure) $\Mcc$ and let $F_1(X), F_2(X)$ be  $\m$-measurable functions on $\Mcc.$ (It is convenient sometimes to assume that $F_\i$ vanish almost everywhere with respect to the measure $\m$.)

In these conditions, we consider the integral operator $\Tb=\Tb[\m,\Kb, F_1,F_2]$ in $L_{2,\m}$
\begin{equation}\label{Tscalar}
    (\Tb u)(X)=\int_{\Mcc}F_1(X)K(X,Y, X-Y)F_2(Y) u(Y)\m(dY),
\end{equation}
or, formally, $(\Tb u)=(F_1\m)\Kb((F_2\m)u).$
More generally, the matrix case in considered, where $K(X,Y,Z)$ is a $\kb\times\kb$ matrix function subject to the above conditions and $F_1(X),F_2(X)$ are $\kb\times\kb$ $\m$-measurable matrix functions; here the operator $\Tb[\m,K, F_1,F_2]$ acts, formally, as \eqref{Tscalar}, but now $u$ is a measurable $\kb$-component vector function on $\Mcc.$ Later we will specify the conditions on the measure $\m$, the kernel $K$ and the weights $F_1,F_2$ granting  boundedness of $\Tb$ in $L_{2,\m}.$ (It is possible in a natural way to consider somewhat more general setting, where $K$, $F_1$, $F_2$ are rectangular matrix-valued functions of proper size, so that the product in \eqref{Tscalar} makes sense, but we avoid this option, leaving it to an interested reader, since it does not involve any new ideas.)

Of special interest is the formally Hermitian case. Here, the kernel $K(X,Y, Z)$ is symmetric in the sense
   \begin{equation*}
    K(X,Y,Z)=K(Y,X,-Z)^*,
   \end{equation*}
where the symbol ${ }^*$ denotes the (complex or, resp.,  matrix) conjugation operation, and $F_1(X)=F_2(X)^*.$
  Under these conditions, the operator $\Tb$ is considered in the space $L_{2,\m};$ it is formally self-adjoint, and if it happens to be bounded, it is self-adjoint in $L_{2,\m}.$

  There are two alternative ways to present results on the properties of integral operators of potential type. On the one hand, one can fix the measure $\m$ and describe how nice the kernel and the weight functions should be. We take a somewhat different course of action: we fix the kernel and determine the properties of the measure and the weight functions that grant the boundedness of the integral operator and the required spectral estimates.

 Let the operator \eqref{Tscalar} be self-adjoint and bounded. In this case, for $\l>0,$ by $n_{\pm}(\l,\Tb)$ we denote the total multiplicity of the spectrum of $\pm\Tb$ in $(\l,\infty)$. If there are infinitely many  such eigenvalues  or there are some points of the essential spectrum in this interval, $n_{\pm}(\l,\Tb)$ is set to  $\infty.$ In the general, not necessarily self-adjoint, case, we study estimates for the \emph{singular numbers} of $\Tb$, i.e., the counting function $n(\l,\Tb):=n_+(\l^2, \Tb^*\Tb).$   We are interested in finding estimates and, if it is possible, asymptotics of these eigenvalue counting functions as $\l\to 0, $ in terms of integral characteristics of the weights $F_1(X), F_2(X),$ of geometric characteristics of the singular measure $\m,$ and, of course, properties  of the kernel $K.$

We restrict ourselves to the case of a compactly supported measure $\m$. Therefore, due to  the standard localization procedure, this problem reduces to the one for operators acting on a set in a smooth compact Riemannian  manifold,  for example, on the torus $\T^\Nb$ with standard metric, the setting we adhere for further on. The homogeneity condition for the kernel is in this case supposed to hold for $Z=X-Y$ close to zero, i.e., near the diagonal $X=Y,$ while the kernel is smooth away from the diagonal. The case of a noncompactly supported measure $\m$ presents certain complications, which we skip in this paper, in order to avoid excessive technicalities.

\subsection{Relation to earlier results }

Spectral problems for weighted weakly polar integral operators were considered in the papers \cite{BS1}, \cite{BS3}, \cite{KostSol}, \cite{BS5}; the main interest being in finding weakest possible conditions on the kernel and  the weights such that the  singular numbers admit the same order estimates as for the non-weighted operator.  In these papers  $\m$ was  the Lebesgue measure on $\R^\Nb$ restricted to the set $\Mcc$. Moreover, an additional factor $\F(X,Y),$ a multiplier, in the kernel of the operator $\Tb$ was present, subject to certain milder regularity restrictions. Is such setting, spectral estimates and asymptotics for $\Tb$ have been established, with order depending on the dimension $\Nb$ and the homogeneity order $\theta.$ For some values of parameters, one of the weights, or both, could be incorporated in the measure, so $\m$ could be any finite Borel measure, not necessarily absolutely continuous with respect to the Lebesgue measure. However, the above spectral estimates had the same order for all admissible weights and measures, and this order depended only on the homogeneity order of the kernel and the dimension $\Nb$ of the space.

  However, certain applications require sharp eigenvalue estimates for
  a more general, singular,  measure $\m.$ For example, $\m$ may be the Hausdorff measure on a fractal set or the surface measure on a nonsmooth surface. In these cases the classical results in the above papers, if applicable, give only '$o$-small' and, therefore, not sharp estimates compared with the absolutely continuous case, so they are   are not applicable, at least directly.
\subsection{Specifics of the new approach and main results}
In the recent papers \cite{RSFA}, \cite{RSLap}, \cite{RConnes}, \cite{RT3}, an approach was developed for obtaining spectral estimates for a class of weighted pseudodifferential operators involving singular measures, the so called Birman-Schwinger type operators. For some cases, in particular for $\m$ being the surface measure on a Lipschitz surface of some positive codimension in $\R^\Nb$, as well as the case of the Hausdorff measure on a uniformly rectifiable set, these estimates are order sharp, which is confirmed by  the  eigenvalue asymptotic formulas  obtained there. In proving these spectral estimates, we use the classical variational approach based upon piecewise-polynomial approximations, initiated by M.Sh. Birman and M.Z. Solomyak more than 50 years ago and  adapted to our singular measure setting. Asymptotic formulas were derived, using perturbation ideas, descendent from  M.Sh. \-Birman and M.Z. \-Solomyak as well; for Lipschitz surfaces, the reasoning was based upon the results by the present authors, \cite{RT1}, \cite{RT2}. One of crucial steps in the reasoning was a reduction of the spectral problem for a pseudodifferential operator with singular weight to an integral operator with singular measure.

In the present paper we apply this relation  to studying the  spectrum  of integral operators with (poly)homogeneous kernel, thus using the above  reduction in the backwards direction. The conditions imposed on the measure $\m$, similarly  to \cite{RT3}, are formulated in terms of inequalities for the $\m$-measure of balls in $\R^\Nb.$ For $0<\a<\Nb,$ we consider three classes of measures:
\begin{equation}\label{a+}
 \mu\in \Pc^\a_{+} \,\, \mbox{if}\,\,  \m(B(X,r))\le \Ac(\m) r^\a,\, r>0;
\end{equation}
\begin{equation}\label{a-}
   \mu\in \Pc^\a_{-} \,\, \mbox{if}\, \m(B(X,r))\ge \Bc(\m) r^\a,\, \, 0<r<\diam \Mcc;
\end{equation}
\begin{equation}\label{a}
  \mu\in\Pc^{\a} \, \mbox{if}\,\, \Bc(\m) r^\a\le \m(B(X,r)) \le \Ac(\m)r^\a, \,\, 0<r<\diam \Mcc;
\end{equation}
here $B(X,r)$ is the (open) ball with radius $r$ and center $X.$

Depending on the homogeneity order of the kernel and the properties of the singular measure   $\m$, we obtain estimates of singular numbers involving certain  integral norms of the weight functions.
Unlike the case of absolutely continuous measures, considered earlier in \cite{KostSol}, \cite{BS5}, and \cite{BS1} where the order of  estimates was determined only by the dimension $\Nb$ and the leading homogeneity order $\theta$ of the kernel,  in our case the order (the power of $\lambda$) of the eigenvalue estimate depends also on the exponent $\a$ in the characteristic \eqref{a+}, \eqref{a-} or \eqref{a}. The choice of which of these conditions is present in  the corresponding formulations is determined by the relation between the dimension $\Nb$ and the order of singularity $\theta$. Note, especially, that when the two-sided condition    \eqref{a} is imposed, a rather special case of the homogeneity degree $\theta=0$ is covered.  Here the leading term in the kernel $K(X,Y,Z)$ is the sum of two terms, one of which, $K_0^{hom},$ is zero order positively  homogeneous in $Z$ for small $Z$ and the other, $K_{0}^{log},$ has the form $Q_0(X,Y)\log |Z|$ with smooth function $Q_0$. In this case, the order of eigenvalue estimates (and asymptotics, when it is proved) equals $-1$ and does not depend on $\a$ in \eqref{a}, the effect noticed earlier in \cite{RSLap}, for the corresponding  spectral problem for pseudodifferential operators, see also the discussion in \cite{RConnes} concerning the noncommutative integration with respect to singular measures.

A special feature of the reduction to  a pseudodifferential problem is the need for smoothness requirements, inherent to the pseudodifferential analysis. For integral operators considered in the paper, such smoothness requirements turn out to be excessive and they are considerably relaxed by means of introduction of  Schur multipliers, similar to  how this was done in the classical papers by M.Sh. \-Birman and M.Z. \-Solomyak for absolutely continuous measures. In proving spectral estimates and asymptotics, we use the additional flexibility provided by the presence of weight functions. When smooth, they can be incorporated in the kernel $K(X,Y,X-Y),$ while, in the opposite, a non-smooth factor in the kernel can be incorporated in the weight functions $F_1,F_2$.  All this enables us to establish the eigenvalue estimates and asymptotics under somewhat milder regularity conditions than in \cite{RT3} in the pseudodifferential setting.

\section{Pseudodifferential and integral operators: reduction}
\subsection{Initial definitions} It is known that an integral operator with homogeneous kernel which is  smooth away from the diagonal, can be understood as a pseudodifferential operator, and vice versa, see, e.g., \cite{BS3}, \cite{RT1}, \cite{Taylor} and references therein.

Having a polyhomogeneous, order $-l<0$, classical symbol $\kF(X,\X)$ admitting
the expansion in homogeneous functions,
    \begin{equation*}
    \kF(X,\X) \sim\sum_{j=0}^{\infty}\kF_j(X,\X), \X\to\infty, \, X\in\R^\Nb, \X\in \R^\Nb\setminus\{0\},
    \end{equation*}

with positively homogeneous terms $\kF_j(X,t\X)=t^{-l-j}\kF_j(X,\X), \, t>0,$
the pseudodifferential operator $\KF(X,D_X)$ with this symbol is defined by the usual local formula
\begin{equation*}
(\KF(X,D_X) u)(X)=\Fc_{\X\to X}^{-1}\kF(X,\X)\Fc_{Y\to\X}u(Y),
\end{equation*}

where $\Fc$ is the Fourier transform.

This operator can be equivalently represented as the
 integral operator with kernel $K(X,Y,X-Y)$ admitting the asymptotic expansion
 \begin{equation}\label{K_j}
    K(X,Y,X-Y)\sim \sum_{j=0}^{\infty} K_{j}(X,X-Y), \, X-Y\to 0
 \end{equation}
in homogeneous or $log$-homogeneous functions.
The kernel $K_j(X,(X-Y))$ is the, properly regularized, Fourier transform of the symbol $\kF_j(X,\X)$
in $\X$ variable. See \cite{RT1},  \cite{Taylor} for details and further references. Note again that the representation \eqref{K_j} is not uniquely defined:   by re-expanding the kernel   $K(X,Y, X-Y)$ at the point $Y=X$  in $X-Y$ variable and further re-grouping  the resulting terms, we arrive at a different composition of the kernel. This relation goes through without complications as long as the symbol and the kernel are infinitely smooth. There is an extensive, rather technically advanced, literature devoted to the treatment of the case of symbols with finite smoothness, see \cite{Abels}, \cite{TaylorBook} and references therein. Our approach, based upon the study of finitely smooth integral kernels is more elementary.

The point of our interest is operators of the form $\Tb=\Tb[\m,\KF,F_1,F_2]$ formally described as $\Tb=(F_1\m)\KF(F_2\m)$,
where $\m$ is a singular measure and $F_1,F_2$ are $\m$-measurable functions. We need to explain how such operators are rigorously defined. The definition will be based upon a fixed factorization of $\KF$ as $\KF=\KF_1\KF_2,$ where $\KF_\i$ are pseudodifferential operators of order $-\g_\i<0,$ $\g_1+\g_2=l.$
The case of our special interest is the one where the pseudodifferential operator $\KF$ is self-adjoint, non-negative, and  has the form $\KF=\LF\LF^*$ where $\LF$ is an order $-\g=-\frac{l}{2}<0$ pseudodifferential operator.

\subsection{Definition of operators}
Following  \cite{RSLap} and \cite{RT3}, we distinguish three cases, where the conditions on the measure $\m$ and the weights $F_\i$ are formulated in different ways. These cases are determined by the  relation between  the order $\g$ and dimension $\Nb$.
\begin{itemize}
\item Subcritical, $2\g_\i<\Nb$; here the measure $\m$ is supposed to belong to $\Pc^\a_+$ with $\a>\Nb-2\g_\i$.
\item Critical, $2\g_\i=\Nb$; here the measure $\m$ is supposed to belong to $\Pc^\a$ with some $\a,$  $0<\a<\Nb$;
\item Supercritical, $2\g_\i>\Nb$; here the measure $\m$ is supposed to belong to $\Pc^\a_-$,  with some $\a>0$.
\end{itemize}

Since $\KF$ is, in fact, an integral operator with polyhomogeneous kernel $K(X,Y,X-Y),$ $\Tb$  can be, still formally,  described as the integral operator
\begin{equation}\label{Operator TKFint}
    (\Tb u)(X)=(\Tb[\m,\KF,F_1,F_2]u)(X)=F_1(X)\int K(X,Y,X-Y)u(Y)F_2(Y)\m(dY)
\end{equation}
We explain now, how this operator is rigorously defined  and in the next section we find conditions for it to be bounded.
Our reasoning, although seemingly arbitrary, follows the  natural rule: if all objects are rigorously  defined, the resulting construction coincides with the the one obtained by formal manipulations.

First, let the pseudodifferential operator $\KF$ be factorized as
    \begin{equation*}
    \KF=\KF_1\KF_2,
    \end{equation*}
where $\KF_1,\KF_2$ are pseudodifferential operators of negative orders, respectively,  $-\g_\i<0,$ $\g_1+\g_2=l.$

Suppose that we have defined \emph{bounded} operators
\begin{equation*}
 \Tb_1=F_1\tb_\Mcc\KF_1: L_2(\T^\Nb)\to L_{2,\m}, \, \Tb_2= {F_2}^*\tb_\Mcc\KF_2^*:L_2(\T^\Nb)\to L_{2,\m},
\end{equation*}

where $\tb_\Mcc$ is the restriction operator of functions in $H^{\g_\i}()$ to $\Mcc.$
  Then the operator $\Tb$ in $L_{2,\m}$
will be defined as
\begin{equation*}
\Tb[\m,\KF,F_1,F_2]=\Tb_1\Tb_2^*.
\end{equation*}

Further on, once we have found \emph{some} singular numbers estimates for $\Tb_\i,$ we can apply the Ky Fan inequality
\begin{equation}\label{KyFan}
 n(\l_1\l_2, \Tb_1\Tb_2^*)\le n(\l_1,\Tb_1)+n(\l_2,\Tb_2)
\end{equation}
 with conveniently chosen $\l_1,\l_2,$ $\l_1\l_2=\l$, to find singular numbers estimates for $n(\l,\Tb)$.

Thus, we are reduced to the task of defining bounded operators $\Tb_\i,$ $\i=1,2,$ and proving their singular numbers estimates.
This topic, actually, has been already covered by the considerations in \cite{MazBook} and further in \cite{RSLap} and \cite{RT3}. In fact,  the boundedness of
$\Tb_\i$ is equivalent to the boundedness of the operator $\Sbb_\i=\Tb_\i^*\Tb_\i$ in $L_2(\T^\Nb).$ The latter operator has the quadratic form
\begin{gather}\label{Sb_1}
    (\Sbb_\i f,f)_{L_2(\T^\Nb)}=(\Tb_\i^*\Tb_\i f,f)_{L_2(\T^\Nb)}=(\Tb_\i f,\Tb_\i f)_{L_{2,\m}}\\\nonumber=\int |F_\i(X)|^2|(\KF_\i f)(X)|^2\m(dX).
\end{gather}
Operators of this type were considered in \cite{RSLap} and \cite{RT3}. There,  conditions have been established for  boundedness of the quadratic form \eqref{Sb_1}, thus justifying the reasoning above, and the spectral estimates for the corresponding operators.
To understand the action of the operator $\Tb_\i$, we consider its sesquilinear form for $u\in L_2(\T^\Nb), v\in L_{2,\m}:$

\begin{equation}\label{TjSes}
    (\Tb_\i u, v)_{L_{2,\m}}=(F_\i \tb_{\Mcc}\KF_\i u, v)_{L_{2,\m}}=(\tb_{\Mcc}\KF_\i u, F_\i^* v)_{L_{2,\m}}.
\end{equation}
For $u\in L_2(\T^{\Nb})$, the function $\KF_\i u $ belongs to the Sobolev space $H^{\g_\i}(\T^{\Nb}).$ So, to assign sense to the expression in \eqref{TjSes}, we need to understand $F_\i^* v$ as an element in the adjoint space $H^{-\g_\i}(\T^{\Nb})$ of distributions. In the supercritical case $2\g_\i>\Nb,$ the space $H^{\g_\i}(\T^{\Nb})$ is embedded in $C(\T^{\Nb})$. Therefore, for $F_\g^*\in L_{2,\m},$ the product $F_\i^*v$ belongs to $L_{1,\m},$ and such function defines a continuous functional on $C(\T^\Nb),$ which means that this is an element in $C(\T^\Nb)'\subset H^{-\g_\i}(\T^{\Nb}).$ In the subcritical case $2\g_\i<\Nb,$ there is no embedding of the Sobolev space $H^{\g_\i}(\T^{\Nb})$ into $C(\T^\Nb),$ so, directly, the restriction of $f=\KF_\i u$ to the support of $\m$ does not make sense.  However,  for the measure $\m$ in $\Pc^\a_+$ with some $\a>\Nb-2\g_\i$ and  for $F_\i^*\in L_{2\s_\i},$ with $2\s_\i=\frac{\a}{2\g_\i-\Nb+\a},$  for any fixed $v\in L_{2,\m}$, the functional
 \begin{equation}\label{Functional}
 \psi_{F_\i^*\overline{v(X)}\m}(f)=\int f(X) (F_\i^*(X)\overline{v(X)}) \m(dX),
 \end{equation}
 defined initially on $H^{\g_\i}(\T^{\Nb})\cap C(\T^\Nb),$ is continuous in $H^{\g_\i}(\T^{\Nb})$-norm by the trace Theorem 11.8  in \cite{MazBook} and thus can be extended by continuity to the whole of $H^{\g_\i}(\T^{\Nb})$.

  A similar reasoning goes through in the critical case $2\g_\i=\Nb.$ There, if $\m$ is in $\Pc^\a_+$ with \emph{some} $\a>0,$ the functional \eqref{Functional}, defined first on \emph{continuous } functions in $H^{\g_\i}(\T^{\Nb}),$ extends to a continuous functional on
on $H^{\g_\i}(\T^{\Nb})$, as soon as $|F_\i^*|^2$ belongs to the Orlicz space $L \log L(\m),$ again by the trace Theorem 11.8 and Corollary 11.8 (2) in \cite{MazBook}. In both  latter cases, by duality, for a fixed $f\in H^{\g_\i}(\T^\Nb),$ the expression \eqref{Functional} defines a continuous antilinear functional of $v\in L_{2,\m},$ therefore, $F_\i \tb_{\Mcc}f,$ thus defined, is an element in $L_{2,\m}.$

In all three cases the operator $\Tb_\i$  can be represented by the diagram
\begin{equation*}
    \Tb_\i: L_2(\T^{\Nb})\overset{\KF_\i}{\longrightarrow}H^{\g_\i}(\Tb^{\Nb})\overset{F_\i\tb_{\Mcc}}{\xrightarrow{\hspace*{1cm}}}L_{2,\m},
\end{equation*}
with all arrows determining continuous operators, therefore, $\Tb=\Tb_1\Tb_2^*$ is a bounded operator in $L_{2,\m}.$

\section{Singular numbers estimates}
\subsection{Singular numbers estimates for $\Tb_\i$}
Here we use the fact that the singular numbers of $\Tb_\i,$ i.e., the eigenvalues of the operator $(\Tb_\i\Tb_\i^*)^{\frac12}$ in $L_{2,\m},$ coincide with eigenvalues of $(\Tb_\i^*\Tb_\i)^\frac12$ in $L_2(\T^{\Nb}).$ The latter operators have been considered in
  \cite{RSLap}, \cite{RT3}. According to the results in these papers, the order  $\l^{-2\s_\i}$ of the spectral estimates for the operator $\Tb_\i=\Tb[\m,\KF_\i,F_\i]$ is determined by the parameters $\g_i,\Nb,\a,$
 namely, $\s_\i=\frac{\a}{2\g_\i-\Nb+\a}$, so, $\s_\i>1$ in the subcritical case, $\s_\i=1$ in the critical case, and $\s_\i<1$ in the supercritical case.
 The assertion to follow is  the combination of Theorem  2.3 in \cite{RSLap} (for the critical case) and Theorems  3.3 and Theorem 3.8 in \cite{RT3} (for the noncritical cases.) Namely, the operator $\Tb_\i^*\Tb_\i,$ having the form $\left((F_\i\m)\tb_{\Mcc}\KF_\i\right)^*\left((F_\i\m)\tb_{\Mcc}\KF_\i\right) ,$ thus defined in $L_2(\T^\Nb)$ by the quadratic form \eqref{Sb_1}, exactly fits in the setting of these theorems, with $V=|F_\i|^2.$ We collect the corresponding results.

 \begin{thm}\label{Th.Basic.Estimate}
 Let the measure $\m$ satisfy conditions \eqref{a+}, \eqref{a} or \eqref{a-}. Suppose that  for the weight function $F_\i$, the function $|F_\i|^2$ belongs to the space $L_{\s_\i,\m}$ in the subcritical case, to $L_{1,\m}$ in the supercritical case, and the Orlicz class $L^{\Psi,\m},$ $\Psi(t)=(1+t)\log(1+t)-t,$ in the critical case. Then the operator $\Tb_\i=\Tb[\m,\KF_\i,F_\i]$ is bounded  as acting from $L_2(\T^\Nb)$ to $L_{2,\m}$ and the following estimates hold for the singular numbers of $\Tb_\i$

\begin{gather}\label{estimate j1}
    n(\l,\Tb_\i)\equiv n(\l^2,\Tb_\i^*\Tb_\i)\lesssim \Cb_{sub}(\m, \KF_\i,F_\i)\l^{-2\s_\i};\,\\\nonumber \Cb_{sub}(\m, \KF_\i,F_\i) =C_\i\||F_\i|^2\|_{L_{\s_\i,\m}}^{\s_\i} ,\, \mbox{\rm{for}} \, \s_\i>1;
    \end{gather}

\begin{gather}\label{estimate j2}
    n(\l,\Tb_\i)\lesssim \Cb_{crit}(\m, \KF_\i,F_\i)\l^{-2};\,\\\nonumber \Cb_{crit}(\m, \KF_\i,F_\i)=C_\i\||F_\i|^2\|^{\Psi,\m,av}, \, \mbox{\rm{for}} \, \s_\i=1;
   \end{gather}
   \begin{gather}\label{estimate j3}
    n(\l,\Tb_\i)\lesssim \Cb_{sup}(\m, \KF_\i,F_\i)\l^{-2\s_\i};\,\\\nonumber \Cb_{sup}(\m, \KF_\i,F_\i) = C_\i\||F_\i|^2\|^{\s_{\i}}_{L_{1,\m}} \m(\Mcc)^{2-2\s_\i}\l^{-2\s_\i}, \, \mbox{\rm{for}} \, \s_\i<1.
\end{gather}
Here, the notation $a(\l)\lesssim b(\l)$ means that $\limsup_{\l\to 0} a(\l)b(\l)^{-1}\le 1$;  in \eqref{estimate j2} $\|.\|^{(\Psi,\m,av)}$ is the averaged  $\Psi$-Orlicz norm with respect to the measure $\m$ (see the definition in \cite{RSLap}, \emph{(2.1)}).
The constant $C_\i$ in \eqref{estimate j1}-\eqref{estimate j3} depends on the dimension $\Nb$, the orders $-\g_\i $ of operators $\KF_\i$,  the characteristic $\a$ of the measure $\m$, on the constants $\Ac(\m),$ $\Bc(\m)$ in \eqref{a-}, \eqref{a}, \eqref{a+}, as well, of course, as on the operator $\KF_\i$, but \emph{not} on the weight function $F_\i$.
\end{thm}

\subsection{Spectral estimates for $\Tb$}

Having the weighted integral  operator $\Tb=\Tb[\m,\KF,F_1,F_2]$ with polyhomogeneous kernel of order $-\theta>-N$ or, what is equivalent, with the pseudodifferential operator $\KF$ of order $-l=-N+\theta<0$, and given weight functions $F_1,F_2$ and measure $\m$, we are free to choose a factorization of the pseudodifferential operator  $\KF=\KF_1\KF_2.$ As $\KF_1$ we can take $\KF_1=(1-\Delta)^{-\frac{\g_1}2},$ therefore, $\KF_2=(1-\Delta)^{\frac{\g_1}2}\KF$; here $\Delta$ is the Laplacian on the torus $\T^N.$ This is equivalent to  factorizing  the integral operator $\Kb$ as $\Kb=\Kb_1\Kb_2$ with orders $\theta_1=-\Nb+\g_1,$ $\theta_2=-\Nb+\g_2.$
 Once this factorization, namely, the choice of the orders $\g_1$ and $\g_2,$ $\g_1+\g_2=\g$, is made, this determines the required properties of $\m,$ $F_1,$ $F_2,$ to be described later on. We arrive at the factorization of our operator $\Tb=\Tb[\m,\KF,F_1,F_1]:$
$\Tb=\Tb_1\Tb_2^*,$ $\Tb_\i=\Tb[\m,\KF_\i,F_\i], \, \i=1,2.$

By the  inequality \eqref{KyFan} for the singular numbers of the product of operators,
estimates \eqref{estimate j1}, \eqref{estimate j2}, \eqref{estimate j3} imply the   estimate for the singular numbers of the operator $\Tb.$
To obtain it, we set in \eqref{KyFan},
 \begin{equation*}
 \l_1=\ab\l^{\frac{\s_2}{\s_1+\s_2}}, \l_2=\ab^{-1}\l^{\frac{\s_1}{\s_1+\s_2}},  \l_1\l_2=\l,
\end{equation*}
with
   \begin{equation*}
    a=\left(\Cb_{*}(\m,\KF_1,F_1) \Cb_{*}(\m,\KF_2,F_2)^{-1} \right)^{\frac{1}{\s_1+\s_2}}.
   \end{equation*}

Then the application of the Ky Fan inequality results in
\begin{equation}\label{Weyl2}
    n(\l,\Tb)\lesssim \Cb(\m,\KF_1,\KF_2, F_1, F_2)\l^{-2\s},
\end{equation}
where
\begin{equation}\label{Weyl3}
   \Cb(\m,\KF_1,\KF_2, F_1, F_2)= \Cb_{
   \divideontimes}(\m,\KF_1,F_1)^{\frac{\s_2}{\s_1+\s_2}}\Cb_{\divideontimes}(\m, \KF_2,F_2)^{\frac{\s_1}{\s_1+\s_2}}
\end{equation}
and $\divideontimes$  stands for the proper subscript in \eqref{estimate j1}-\eqref{estimate j3}.

From the expression for $\s_1,\s_2$ we see that the exponent $-2\s={-2\frac{\s_1\s_2}{\s_1+\s_2}}$ in \eqref{Weyl2} equals, in fact $-\frac{2\a}{l+2\a},$ which means that the order in the singular numbers estimate does not depend on the way how the factorization of $\KF$ is chosen. This choice, however, determines the conditions imposed on the measure $\m$ and the weight functions $F_1,F_2.$

In the Theorem to follow, we collect possible combinations of the orders in the factorization and describe the corresponding conditions for $\m$ and $F_\i,$ $\i=1,2.$
\begin{thm}\label{Thm.Est.fact}Let $\KF$ be a pseudodifferential operator of order $-l<0$ factorized as $\KF=\KF_1\KF_2,$ where $\KF_\i,$ $\i=1,2$ are pseudodifferential operators of order $-\g_\i<0,$ $\i=1,2,$ $\g_1+\g_2=l.$  Then for the operator $\Tb=\Tb[\m,\KF,F_1,F_1]$ the singular numbers estimate \eqref{Weyl2} holds with $\Cb(\m,\KF_1,\KF_2, F_1, F_2)$ given by \eqref{Weyl3},  the following conditions being imposed on $\m,F_\i$:
\begin{enumerate}
\item if $2\g_1,2\g_2<\Nb $ then  $\m\in \Pc^{\a}_+,$ $2\a>\Nb-2\g_\i,$ and $F_\i\in L_{2\s_i,\m},$ $\s_\i=\frac{\a}{2\g_i-\Nb+2\a}$;
    \item if $2\g_1<\Nb,$ $2\g_2=\Nb$ then $\m\in\Pc^{\a},$ $2\a>\Nb-2\g_1,$ and  $F_1\in L_{2\s_1,\m},$ $\s_1=\frac{\a}{2\g_1-\Nb+2\a},$ $|F_2|^2\in L^{\Psi,\m}$;

        \item if $2\g_1<\Nb,$ $2\g_2>\Nb$ then $\m\in \Pc^\a,$ $2\a>\Nb-2\g_1$ and $F_1\in  L_{2\s_1,\m},$ $F_2\in L_{2,\m}$;
        \item if $2\g_1=\Nb$, $2\g_2=\Nb$ then $\m\in\Pc^\a,$ $\a>0,$ and $|F_\i|^2\in L^{\Psi,\m},$ $\i=1,2$;
       \item if $2\g_1=\Nb,$ $2\g_2>\Nb$ then $\m\in \Pc^\a,$ $\a>0$ and $|F_1|^2\in  L^{\Psi,\m},$ $F_2\in L_{2,\m}$;
        \item if $2\g_1,2\g_2>\Nb$ then $\m\in \Pc^{\a}, \a>0$ and $F_\i\in L_{2,\m},$ $\i=1,2.$
\end{enumerate}
\end{thm}

\subsection{Estimates for lower order operators}
In our analysis, further on, we will need  singular numbers estimates for the case when the pseudodifferential operator $\KF$ is replaced by an operator of a lower order.

Let now, for the weights $F_1,F_2$ satisfying the conditions of Theorem \ref{Th.Basic.Estimate} with a factorization as in Theorem \ref{Thm.Est.fact} for certain fixed measure $\m$ and the order $-l<0,$ we replace the order $-l$ pseudodifferential operator $\KF$ by another one, $\KF',$ of order $-l'<-l$. It is natural to expect that the singular numbers of the operator $\Tb'=\Tb[\m,\KF',F_1,F_2]$ decay faster than the ones  of $\Tb=\Tb[\m,\KF,F_1,F_2]$.  This property is not quite trivial since it may happen that while $\Tb$ (or some of its factors) belongs to the subcritical case, the operator $\Tb'$ may get into the critical or supercritical case, so the conditions imposed on the measure $\m$ and the weights $F_\i$ might change. The following corollary states that, nevertheless,  certain, probably, not sharp, spectral estimates hold without changing the conditions for the measure $\m.$
\begin{cor}\label{Cor.lower.order} Let the conditions of Theorem \ref{Th.Basic.Estimate}, \ref{Thm.Est.fact} be satisfied for a certain factorization of $\KF$ and certain $\m, F_1, F_2, \KF$. Let $\KF'$ be an order $-l'<-l$ pseudodifferential operator. Then
\begin{equation}\label{Lower}
    n(\l,\Tb')=o(\l^{-2\s}), \, \l\to 0, \, \Tb'=\Tb[\m,\KF',F_1,F_2].
\end{equation}
\end{cor}
\begin{proof}Let $\KF=\KF_1\KF_2$  be the factorization present in Theorem \ref{Th.Basic.Estimate}, with operators $\KF_1,\KF_2$ having order $-\g_1,-\g_2, \,\g_1+\g_2=l$, so that $\m,F_1,F_2$ satisfy the conditions of this theorem. We factorize $\KF'$ in the following way
\begin{equation}\label{LowOrder}
    \KF'=\KF_1'\KF_3'\KF_2'.
\end{equation}
Here we take $\KF_1'=(1-\Delta)^{-\g_1/2}, \KF_2'=(1-\Delta)^{-\g_2/2}$ and $\KF'_3=(1-\Delta)^{\g_1/2}\KF'(1-\Delta)^{\g_2/2}.$
The factorization \eqref{LowOrder} leads to the factorization of the operator $\Tb'$:
   \begin{equation*}
    \Tb'=\Tb_1 \KF_3' \Tb_2^*, \, \Tb_1=F_1\KF_1', \Tb_2=F_2^*\KF_2'.
   \end{equation*}

Singular numbers estimates for $\Tb_1$ and $\Tb_2$ are the same as in Theorem \ref{Thm.Est.fact}, while $\KF_3'$, is a pseudodifferential operator on $\T^\Nb$ of negative order $(l-l')<0,$ having decaying singular numbers, $n(\l, \KF_3')=O(\l^{(l-l')/\Nb})$.
By the Ky Fan inequality, it follows that \eqref{Lower} holds.
\end{proof}
\begin{remark}
 One might have  considered a more general setting, namely, an operator of the form \eqref{Operator TKFint} but acting from $L_{2,\m}$ to $L_{2,\m'},$ where $\m'$ is a different (possibly, singular, measure.) Here, extensive complications arise since the spectral estimate should now depend not only on the properties of the measures $\m,\m'$ taken separately, but also on their relative position in $\R^\Nb.$ We leave this line of studies for a time being.
\end{remark}
\section{Schur multipliers and spectral estimates}
In papers by M.Sh. Birman and M.Z. Solomyak, a deep study was performed, concerning  properties of transformations in Schatten classes of integral operators, generated by Schur multipliers. Namely, having a function $\F(X,Y)$, the Schur multiplier transformation $\Mb[\F]$ associates with an integral operator with kernel $K(X,Y)$ another operator, the integral operator with kernel $\F(X,Y)K(X,Y)$. Explicit analytic conditions for the function  $\F$ granting that $\Mb[\F]$  transforms any integral operator in a certain Schatten class to an operator in the same class (or in some other prescribed Schatten class) were found. In particular, in \cite{KostSol}, the results on multipliers in weak Schatten classes were used for obtaining eigenvalue estimates for operators with kernel of the form $\F(X,Y)K(X,Y, X-Y),$ where $K$ is an integral kernel with
weak singularity as $X-Y\to 0.$ Even more general results have been obtained in \cite{BS5} (Sect. 8,10)
and in \cite{Kost.MSb}. In these papers, operators were considered in $L_2$ spaces with measure. For kernels with a relatively strong singularity (of order $-\theta \le -\Nb/2$) at the diagonal, these measures were supposed to be absolutely continuous with respect to the Lebesgue measure, with prescribed properties of the densities, while for a weaker singularity the results on singular numbers estimates hold for any finite measures, thus admitting the singular ones.  Simultaneously, the weaker is the polarity of  the kernels of integral  operators  at the diagonal, the more smoothness is required for the multiplier $\F(X,Y)$ to generate a transformation in the proper Schatten class. The conditions imposed of $\F(X,Y)$ in these papers are expressed in rather complicated terms. We give here a simplified version, sufficient for our needs.
\begin{proposition}\label{Prop.Multipl.Est} Let $\F(X,Y)$ be a function on a  $\Qb\times\Qb$, where $\Qb$ is a cube in $\R^N$. Then there exists $\mb(\Nb,q)$ such that for $m>\mb(\Nb,q)$,
any function $\F(X,Y)$ possessing continuous partial derivatives of order up to $m$ in all variables is a Schur multiplier in the space  of integral operators $\Tb: L_2(\Qb,\m_2)\to L_2(\Qb,\m_1) $ in the weak Schatten class $\Sigma_q$ (of operators with singular numbers satisfying $n(\l,T)=O(\l^{-\frac{1}{q}})$) for  \emph{any } finite measures $\m_1,\m_2.$
\end{proposition}

We will also use the fact that, for any finite measures $\m_1,\m_2,$ operators with smooth kernel have arbitrarily rapid eigenvalues decay rate, as soon as the smoothness is sufficiently high. This statement is contained as a particular case of Proposition 2.1 in \cite{BS5}, where more complicated but less restrictive condition are imposed.
\begin{proposition}\label{Prop.smooth}Let $\m_1,\m_2$ be finite Borel measures on a cube $\Qb\subset \R^\Nb$. Suppose that $U(X,Y)\in C^m(\Qb\times \Qb)$ and $m>2\Nb$. Then the operator $\Ub:L_{2,\m_2}\to L_{2,\m_1},$
    \begin{equation*}
    (\Ub u)(X)=\int U(X,Y)u(Y)\m_2(dY)
    \end{equation*}
has singular numbers satisfying
    \begin{equation*}
    n(\l,\Ub)=O(\l^{-{1/q}}), \, q=\frac{\Nb+2m}{2\Nb}.
    \end{equation*}

\end{proposition}
It stands to reason that these results carry over automatically to operators defined on the torus $\T^{\Nb}$ instead of a cube.
In our applications, the role of $\m_1,$ $\m_2$ will be played by $F_1\m$,  $F_2\m,$ where $\m$ is a singular measure satisfying one of conditions in \eqref{estimate j1}--\eqref{estimate j3} and $F_1,F_2$ are weight functions. Thus, using Proposition \ref{Prop.smooth}, we arrive at our result on multipliers in the set of operators with singular measures.\\
We denote by $\Ks[\Nb,\theta,\a]$ the space of integral operators satisfying one of conditions of Theorem \ref{Th.Basic.Estimate}.
\begin{thm}\label{Thm.multiplic}
Let $\Kb$ be an integral operator in $\Ks[\Nb,\theta,\a]$ with some weakly polar kernel $K(X,Y,X-Y)$ with singularity order $-\theta$ at the diagonal, with the  measure $\m$ and the weight functions $F_1,F_2$ as above. Suppose that $\F$ is a function on $\T^\Nb\times\T^\Nb$ belonging to $C^{2m}(\T^\Nb\times\T^{\Nb}),$ $m>\mb(\Nb,2\s)$. Then for the integral operator $\Hb$ with the same measure $\m$ and weight functions and with the kernel $H(X,Y)=\F(X,Y)K(X,Y,X-Y),$  the singular numbers  estimate holds
  \begin{equation*}
   n(\l,\Hb) \lesssim  C(\F)\Cb(\m,\KF_1,\KF_2, F_1, F_2) \l^{-2\s}
  \end{equation*}

with the  constant $\Cb(\m,\KF_1,\KF_2, F_1, F_2)$ is determined by the parameters in Theorem  \ref{Thm.Est.fact}, thus depending  on the factorization of the operator $\KF$, the parameter of the measure $\m$ and the proper integral norms of the weight functions $F_1,F_2$, while $C(F)$ depends on the bounds for the derivatives of $\F.$
\end{thm}
\begin{proof} For every fixed $X\in\Mcc,$ we consider the starting fragment of the Taylor expansion of the function $\F(X,Y)$ at the  point $(X,X)$ in powers of $Y-X$:
\begin{gather}\label{Taylor}
    \F(X,Y)=\sum_{|\nb|<m}(\nb!)^{-1}\partial_Y^\nb\F(X,X)(Y-X)^{\nb} +\F_{(m)}(X,Y)\\\nonumber\equiv\sum_{|\nb|<m} \F_{\nb}(X)(Y-X)^{\nb}+\F_{(m)}(X,Y).
\end{gather}
We consider first the leading term in $\Hb$ corresponding to the first term in \eqref{Taylor}, $\nb=0$,
 \begin{equation*}
 H_{0}(X,Y,X-Y)=\F(X,X)   K(X,Y,X-Y)\equiv \F_0(X)   K(X,Y,X-Y).
 \end{equation*}

For given $F_1,F_2,$ the integral operator $\Hb_0$ with kernel $H_{0}(X,Y,X-Y)$ acts as
\begin{equation*}
   (\Hb_0 u) (X) = \int F_1(X)\F(X,X)  K(X,Y,X-Y) u(Y)F_2(Y)\m(dY).
\end{equation*}

This operator is of the same kind  as our initial one, however with the weight function $F_1$ replaced by the weight function $F_1(X)\F(X,X).$   Note also that $\F(X,X)$ is a bounded function, and therefore $F_1(X)\F(X,X)$ belongs to the same space of weights as $F_1,$ the space required by Theorems \ref{Th.Basic.Estimate}, \ref{Thm.Est.fact}. Therefore, incorporating $\F(X,X)$ into the weight function, we obtain  for the operator $\Hb_0$ the same singular numbers estimate.

Next, we consider the operator  $\Hb_{\nb}$ corresponding to some term in \eqref{Taylor} with $|\nb|>0.$ 
 The kernel $H_{\nb}(X,Y,X-Y)=K(X,Y,X-Y)(X-Y)^\nb$ is smooth for $X\ne Y$ and has homogeneity of order $\theta+|\nb|$ in $X-Y$, so it is  larger  than the homogeneity order which $K$ has. Since the function $\F_{\nb}(X)$ is bounded, the weight function $F_1(X)\F_{\nb}(X)$ belongs to the same  space of weights as $F_1.$ By Corollary \ref{Cor.lower.order}, this gives us the spectral estimate
    \begin{equation*}
    n(\l,\Hb_{\nb})=o(\l^{-\frac{2\a}{2l-N+2\a}}).
    \end{equation*}

Finally,  we consider the remainder term for $\Hb,$ namely, $H_{(m)}(X,Y,X-Y)=\F_{(m)}(X,Y)  K(X,Y,X-Y),$ in \eqref{Taylor}. By our assumptions, the function $\F_{(m)}$ is $C^\mb$ smooth outside the diagonal $X=Y$, therefore, the same is valid for the product $\F_{(m)}K$. As for the diagonal, $X=Y,$ the function $\F_{(m)}(X,Y)$ has zero of order $m$ as $X\to Y$; therefore, the product $H_{(m)}=\F_{(m)}K$ has zero of order not less than $\Nb+|\theta|$, and
Proposition \ref{Prop.smooth} gives the required singular number estimate, as soon as $m$ is taken large enough.
\end{proof}

In the matrix case, we can consider both left and right multipliers:
\begin{equation*}
\Mb[\F_\ell,\F_r]K=\F_\ell(X,Y)K(X,Y,X-Y)\F_r(X,Y),
\end{equation*}

with matrix functions $\F_\ell(X,Y)$, $\F_r(X,Y)$ of proper size. The reasoning above carries over to this case automatically.

Note here that we used the arbitrariness in the choice of the weights in the proof above: we could incorporate a 'not sufficiently smooth' factor in the kernel into the weight function. We will return to this pattern further on, when we consider the multipliers in the asymptotic formulas.

\section{Eigenvalue asymptotics}\label{asymptotics} Following the basic strategy of  using perturbation  approach, the asymptotic formulas for eigenvalues can be established. We consider the self-adjoint case, so the entries in the operator $\Kb$ in \eqref{Tscalar} satisfy the following conditions:
  \begin{equation*}
    K(X,Y,X-Y)=K(Y,X,Y-X)^*; \, F_1(X)=F_2(X)^*;
  \end{equation*}

here the symbol ${}^*$ denotes the complex conjugation in the scalar case and the matrix conjugation in the vector case.

We suppose that the Hermitian kernel $K(X,Y,X-Y)$ is smooth for $Y\ne X$ and admits the asymptotic expansion in homogeneous functions, as in Sect.2, with the leading homogeneity order $-\theta>-\Nb.$ Let $\G\subset \T^{\Nb}$ be a Lipschitz surface in $\T^{\Nb}$ of dimension $d<\Nb$ and codimension $\dF=\Nb-d.$ As measure $\m$, we take the Hausdorff measure $\Hc^{d}$ on the surface $\G,$ so, $\Mcc=\G.$ Locally, let the surface be described by the equation $y=\vf(x),$ in co-ordinates $X=(x,y)$, such that $x
\in \Uc\subset \R^d$, $y\in\R^{\dF}$ and $\vf$ is a vector-function with $\dF$ components. In these co-ordinates, the measure $\m$ is described by $\m(dx)=\det(\pmb{1}+(\nabla\vf)^*(\nabla\vf))^{\frac12}dx.$ Such measure  $\m$ belongs to $\Pc^{\a},$ $\a=d$ with constants $\Ac(\m),\Bc(\m)$ determined by the surface $\G$ globally.  We suppose that the kernel $K(X,Y,X-Y)$ with leading homogeneity order $-\theta$ satisfies the conditions of Theorems \ref{Th.Basic.Estimate}. We suppose also that  the weight function $F=F_1=F_2^*$ satisfies the conditions of
Theorem \ref{Thm.Est.fact} under some factorization as well.

In order to write down the asymptotic formulas for eigenvalues,
it is convenient to use the pseudodifferential representation: the integral operator with kernel $K(X,Y,X-Y)$ in $\R^\Nb$ is a pseudodifferential operator $\KF$ of order $-l=-\Nb+\theta.$ Following \cite{RT3}, we suppose that $l>\dF.$
The principal symbol of $\KF,$ $\kF_0(X,\X),$  is the regularized Fourier transform of the leading term $K_0(X,X,X-Y)$ of the kernel in the last variable, see details in \cite{BS3},\cite{BS5},\cite{RT1} and Sect.2 above.

We recall here the expression for the coefficient in the asymptotic formula, see \cite{RT2}, \cite{RT3}.

By the Rademacher theorem, at $\Hc^d$-almost all points $X$ on $\G,$ there exists the tangent $d$-dimensional space $\mathrm{T}_X\G.$ We identify it with the cotangent space $\mathrm{T}_X^*\G.$ Similarly, $\mathrm{N}_X\G$ denotes the normal $\dF$-dimensional space to $\G$ (identified  with the conormal one.) For such points $X,$ we define the symbol of order $-l+\dF<0$ on $\G$:
\begin{equation*}
    \rb_{0}(X,\x)=(2\pi)^{-\dF}\int_{\mathrm{N}_X\G}\kF_0(X;\x,\eta)d\eta,\, (X,\x)\in \mathrm{T}_X^*\G,
\end{equation*}
and the density
\begin{equation}\label{density}
    \r_{\KF}^{\pm}(X)=\int_{\mathrm{S}_X\G}|F(X)|^{2\s}\rb_{0}(X,\x)_{\pm}^{\s}d\x,\,\s =\frac{d}{l-\dF},
\end{equation}
where $\rb_{0}(X,\x)_{\pm}$ are the positive, resp., negative part of the symbol $\rb_{0}(X,\x),$
    \begin{equation*}
    \rb_{0}(X,\x)_{\pm}=\frac12\left(|\rb_{0}(X,\x)|\pm\rb_{0}(X,\x)\right).
    \end{equation*}

The eigenvalue asymptotics result for a \emph{smooth} kernel is the following.
\begin{thm}\label{Th.As.smooth}Under the above conditions, for the eigenvalues of $\Tb=\Tb[\m,\KF,F,{F}^*],$ the asymptotic formulas hold
\begin{equation}\label{As.smooth}
    n_{\pm}(\l,\Tb)\sim \l^{-2\s}\Ab_{\pm}, \,\l\to0,
\end{equation}
with coefficient
\begin{equation}\label{Coeff}
    \Ab_{\pm}=\frac{1}{d(2\pi)^{d-1}}\int_{\G}\r_{\KF}^{\pm}(X)\m(dX).
\end{equation}
\end{thm}
\begin{proof}The proof is similar to the reasoning in Theorem 6.2 in \cite{RT3} or Theorem 6.4 in \cite{RT2}; we explain the main steps, not going into technical details.

Suppose first that the   weight function $F$ is the restriction to  $\G$ of a function $\tilde{F}$  defined and smooth in $\T^\Nb$. Then we can incorporate $F$ and $F^*$ into the kernel $K,$ keeping it smooth, and now  the result is contained in Theorem 6.4 in \cite{RT2}.  Further, for a general weight $F$ satisfying the conditions of the Theorem, we can approximate it in the proper integral norm on $\G$ (namely, $L_{2\s,\m}, L^{2\Psi,\m}, L_{2,\m},$ depending on the case in Theorem \ref{Thm.Est.fact}) by a function $F_{(\e)}$ admitting  an extension as a smooth function in $\T^\Nb.$ The construction of such approximation is described in Lemma 6.1 in \cite{RSLap}. This approximation, by our Theorem 3.1, leads to the smallness of the coefficient in the asymptotic eigenvalue estimates for the difference $\Tb[\m,\KF,F,F^*]-\Tb[\m,\KF,F_{(\e)},F_{(\e)}^*].$ By the standard application of the asymptotic perturbation lemma,
see, e.g., Theorem 4.1 in \cite{BS2}, or Lemma 6.1 in \cite{RT2}, we pass to the limit in the eigenvalue asymptotic formula for $F_{(\e)}$ as $\e\to 0,$ which gives \eqref{Th.As.smooth}.
\end{proof}
\begin{remark}\label{faulty way}One might be tempted to refer directly to the eigenvalue asymptotics results in \cite{RT3}. However, the situation is not that simple. The eigenvalue asymptotics in \cite{RT3} was proved for operators of the form $\AF^*(V\m)\AF,$ where $\AF$ is a negative order pseudodifferential operator and $\m$ is a singular measure. Passing to integral operators, as in the present paper, we arrive at the study of eigenvalues  for $(V^{\frac12}\m)\AF\AF^*(V^{\frac12}\m)$; here  $\KF=\AF\AF^*$ is a \emph{nonnegative} pseudodifferential operator.  Therefore the analysis of the operator $(V^{\frac12}\m)\KF(V^{\frac12}\m)$ when $\KF$ is \emph{not sign-definite} does not pass through. Therefore we need to refer to the results in \cite{RT2} and perform an additional approximation.
\end{remark}
\begin{remark}\label{Rem.Matrix}The approach we use in this Theorem can be applied for operators acting in the spaces of vector-functions. While, formally, the required results  in \cite{RT1}, \cite{RT2}, \cite{RT3} are not presented for  the vector case, they follow automatically, without involving additional consideration. What is important here is that the basic results  on the eigenvalue asymptotics  for negative order pseudodifferential operators are already stated  and proved in \cite{BS5} for the vector case \emph{(}see also \cite{Ponge}\emph{)}.
The asymptotic formulas for the vector case coincide with
\eqref{As.smooth}, however the expressions \eqref{Coeff} for the  asymptotic coefficients $\Ab^{\pm}$ should be replaced by their matrix versions.

Namely, for the matrix case, the expression for the density in \eqref{density} should be replaced by
\begin{equation}\label{density.matrix}
    \r_{\KF}^{\pm}(X)=\int_{\mathrm{S}_X\G}\Tr\left\{\left[F(X)\rb_{0}(X,\x)F^*(X)\right]_{\pm}^{\s}\right\}d\x,
\end{equation}
with further calculation of the coefficients $\Ab^{\pm}$ by means of \eqref{Coeff}. In \eqref{density.matrix} the subscript $\pm$ means now the positive, resp., negative part of the corresponding matrix, so the expression on the right means that first the proper (positive or negative) part of the hermitian matrix  $F(X)\rb_{0}(X,\x)F^*(X)$ is found, then it is raised to the power $\s$ and then the trace of the resulting matrix is calculated to be further integrated over the cotangent sphere $\mathrm{S}_X\G$.
\end{remark}

\section{Multipliers in the eigenvalue asymptotics}
In Section \ref{asymptotics} above, the conditions imposed on the kernel $K(X,Y,X-Y)$ required it to be infinitely smooth for $X\ne Y$. The following reasoning involving multipliers allows us to reduce these smoothness conditions. Our results are not optimal since the regularity conditions imposed on the multiplier $\F$ may be weakened using the full strength of multiplier theory by M.Sh. Birman and M.Z. Solomyak. We restrict ourselves to a rather technically  simple setting to keep the paper more elementary.

Let $K(X,Y, X-Y)$ be a polyhomogeneous kernel of order $-\theta>-\Nb,$ smooth for $X\ne Y$. Suppose that $\G$ is a Lipschitz surface in $\T^\Nb$ of dimension $d\ge1$ and codimension $\dF\ge 1$ with surface measure $\m$. Let $F(X),$ $X\in\G$ be a weight function satisfying the condition of Theorem \ref{Th.As.smooth}. The function  $\F(X,Y)$, the multiplier, is supposed be a function in $C^{2m}$ in a neighborhood of $\G\times\G,$ $2m>2\mb+|\theta|.$
Consider the transformed, also Hermitian, kernel
\begin{equation*}
 K_{\F}(X,Y,X-Y)=\F(X,Y) K(X,Y,X-Y)\F(Y,X)^*.
\end{equation*}

 \begin{thm}\label{Mult.Asymp.Th}Under the above conditions, for the integral operator
   \begin{equation*}
    \Kb_{\F}: u(X)\mapsto \int_{\G}F(X)K_{\F}(X,Y,X-Y)F(Y)^* u(Y)\m(dY),
   \end{equation*}

 the eigenvalue asymptotics \eqref{As.smooth} holds
 with the coefficient $\Ab_{\pm}$ calculated by means of \eqref{Coeff},
 with $F(X)$ replaced by $F(X)\F(X,X).$
 \end{thm}
 \begin{proof} The reasoning follows the pattern used in Theorem \ref{Thm.multiplic}.
 Consider the starting fragment of the Taylor  expansion of the multiplier  $\F(X,Y)$ in $X-Y$ variable at the point $Y=X,$
 \begin{equation}\label{Taylor F}
 \F(X,Y)=\F(X,X)+\sum_{1\le|\nb_Y|\le m}
 (\nb_X!)^{-1}\F_Y^{(\nb_Y)}(X,X)(Y-X)^{\nb_Y}+\F_{(m)}(X,Y),
 \end{equation}
 where  the remainder $\F_{(m)}(X,Y)$ is a function in $C^{m}$ satisfying
 $\F_{(N)}(X,Y)=o(|X-Y|^{m+|\theta|}),$ as  $X-Y\to 0.$
 Similarly, $\F(Y,X)^*$ expands as
 \begin{equation}\label{Taylor F*}
 \F(Y,X)^*=\F(Y,Y)^*+\sum_{1\le|\nb_X|\le m}
 (\nb_X!)^{-1}\F_X^{(\nb_X)}(Y,Y)^*(X-Y)^{\nb_X}+\F_{(m)}(Y,X)^*
 \end{equation}
 Using \eqref{Taylor F}, \eqref{Taylor F*}, we represent $K_{\F}(X,Y,X-Y)$ as the sum of three terms
 \begin{equation}\label{KF 3 terms}
  K_{\F}(X,Y,X-Y)=\F(X,X)K(X,Y,X-Y)\F(Y,Y)^*+ K_{\F}^{(m)}+\widetilde{K_{\F}^{(m)}},
 \end{equation}
 where
 \begin{gather}\label{KFN}
    K_{\F}^{(m)}(X,Y,X-Y)=\\\nonumber \sum_{\overset{|\nb_X|,|\nb_Y|\le m,}{ |\nb_X|+|\nb_Y|>0}}
    \frac{1}{\nb_X!}\frac{1}{\nb_Y!}\F_Y^{(\nb_Y)}(X,X)(Y-X)^{\nb_Y}
    K(X,Y,X-Y)(\F^*)_X^{(\nb_X)}(Y,Y)(Y-X)^{\nb_X},
 \end{gather}
 and $\widetilde{K_{\F}^{(m)}}$ contains the remainder terms in the expansions \eqref{Taylor F}, \eqref{Taylor F*}.
 The first term in \eqref{KFN}, the  product  $F(X)\F(X,X)K(X,Y,X-Y)\F(Y,Y)^*F^*(Y)$ can be re-grouped as $(F(X)\F(X,X))K(X,Y,X-Y)(\F(Y,Y)^*F(Y)^*).$ This means that the corresponding integral operator can be considered as the one with the same smooth kernel $K(X,Y,X-Y)$  but with a different weight function $F(X)\F(X,X)$  instead of $F(X).$ Since the function $\F(X,X)$ is bounded, the new weight, $F(X)\F(X,X)$, belongs to the same  space of integrable functions as $F(X).$ Therefore, to the integral operator under consideration, the eigenvalue asymptotics theorem can be applied, which gives the declared expression for the asymptotic coefficients.

 In the second term in \eqref{KF 3 terms}, a single summand, up to a constant, has the form
 \begin{equation}\label{one term}
    \F_Y^{(\nb_Y)}(X,X)(Y-X)^{\nb_Y}
    K(X,Y,X-Y)(\F^*)_X^{(\nb_X)}(Y,Y)(X-Y)^{\nb_X}
 \end{equation}
 Here, the product $K(X,Y,X-Y)(Y-X)^{\nb_Y+\nb_Y}$ is smooth for $X\ne Y$ and has a weaker singularity than $K$ as $X-Y\to 0$. The derivatives $\F_Y^{(\nb_Y)}(X,X)$, $\F_X^{(\nb_X)}(Y,Y)$ are bounded and can be incorporated in the weight functions $F(X)$ and $F^*(Y)$. Therefore, for the corresponding integral operator, by Corollary \ref{Cor.lower.order}, the singular numbers estimate holds, with a faster eigenvalues decay rate.

 Finally, in the third term in  \eqref{KF 3 terms}, in each summand, the remainder
 in the Taylor expansion of the multiplier $\F$ is present. Such remainder has zero of high order at $X=Y,$ together with derivatives. Therefore, the product of $K(X,Y,X-Y)$ with such remainder has bounded derivatives of sufficiently high order  and by Proposition \ref{Prop.smooth}, the corresponding operator  has arbitrarily fast decaying  singular numbers, as soon as the order of the derivatives is sufficiently high.
 \end{proof}
 \begin{remark}\label{multi.matrix} Similar to Remark \ref{Rem.Matrix},
 the above result on the multipliers in eigenvalue asymptotics can be automatically extended to the vector case. Here the multiplier $\F(X,Y)$ is a $\kb\times\kb$ matrix and the transformed kernel is
 \begin{equation}\label{Mult.kernel}
  K_\F(X,Y,X-Y)=\F(X,Y)K(X,Y,X-Y)\F(Y,X)^*.
 \end{equation}
 The reasoning in proving the asymptotic formula for eigenvalues is the same as in Theorem \ref{Mult.Asymp.Th}.
 \end{remark}

\end{document}